\documentclass[12pt]{article}
\usepackage{enumerate}
\usepackage{amsmath}
\usepackage{amsthm}
\usepackage{amssymb}
\usepackage{algorithm}
\usepackage[bookmarksnumbered,  plainpages]{hyperref}

\newtheorem{thm}{Theorem}[section]

\newtheorem{lem}{Lemma}[section]

\newtheorem{prop}[thm]{Proposition}
\theoremstyle{definition}

\theoremstyle{remark}
\newtheorem{rem}{Remark}[section]
\begin{document}
	
	\title{{More on Equivalent Formulation of  Implicit Complementarity  Problem}}
	
	\author{Bharat kumar$^{a,1}$, Deepmala$^{a,2}$ and A.K. Das $^{b,3}$\\
		\emph{\small PDPM-Indian Institute of Information Technology, Design and Manufacturing,}\\
		\emph{\small Jabalpur - 482005 (MP), India}\\
		\emph{\small $^{b}$Indian Statistical Institute, 203 B.T. Road, Kolkata - 700108, India }\\
		\emph{\small $^1$Email:bharatnishad.kanpu@gmail.com , $^2$Email: dmrai23@gmail.com}\\		\emph{\small $^3$Email: akdas@isical.ac.in}}
	\date{}
	\maketitle
	
	\abstract{\noindent  This article presents an equivalent formulation  of the implicit complementarity problem. We  demonstrate that solution of the equivalent  formulation is equivalent to the solution of the implicit complementarity problem. Moreover, we provide another equivalent formulation of the implicit complementarity problem using a strictly increasing function.
	 }

	
	
	
	\maketitle
	
	\section{Introduction}\label{sec1}
Bensoussan et al. \cite{Ben1975} presented the implicit complementarity problem (ICP).  	 ICPs are a class of mathematical optimization problems that solve a system of nonlinear equations that includes both complementary conditions and equality or inequality constraints. ICPs can be formulated as follows: \\Consider the matrix  $\mathcal{A} \in \mathbb{R}^{n\times n}$  and  the vector  $b \in \mathbb{R}^n$,  the implicit complementarity problem denoted as ICP$( \mathcal{A}, b, f )$ is to find the solution $ r \in \mathbb{R}^n $ to the following system:
\begin{equation}\label{eq0}
	\mathcal{H}(r)=r-f(r)\geq 0,~~~  \mathcal{F}(r)=\mathcal{A}r+b \geq 0,~~~ \mathcal{H}(r)^T\mathcal{F}(r)=0,	
\end{equation}
where $f(r)$ is a mapping from $\mathbb{R}^n$ to $\mathbb{R}^n$.
ICPs arise in many different areas of mathematics and science, including economics, engineering, physics and computer science. For details see,  \cite{Lemke1964},  \cite{Ferris2011} and \cite{Murty1988}.\\

\noindent One of the most popular techniques for developing fast and affordable iterative algorithms is the equivalent formulation of the LCP$(\mathcal{A}, b)$ as an equation that's solution must be the same.  Mangasarian offered an equivalent forms    of  LCP$(\mathcal{A}, b)$  in \cite{Mangasarian1977} and described as 
\begin{equation*}
	r = (r- \omega \Omega(\mathcal{A}r + b))_+,
\end{equation*}
where  $r_+ \in \mathbb{R}^n$, $(r_+)_i = max\{0, r_i\}$ and  $\Omega \in \mathbb{R}^{n\times n}$ is a positive diagonal matrix.   The LCP $(\mathcal {A}, b)$ is described in an analogous form and several iteration techniques are given by Bai in \cite{Bai2010}. For more details on equivalent form of  LCPs and related iteration methods see, \cite{Ahn1981}, \cite{Hong2016}, \cite{Kumaron}, \cite{Kumarmore}, \cite{Noor1988}, \cite{Kumarpro}, \cite{Kumarnote} and  \cite{Cottle1992}. The concept of equivalent formulation has also been used effectively for other complementarity problems, like  implicit complementarity problem \cite{Xie2016} and  \cite{Hong2016} and  horizontal linear complementarity problem \cite{Mezzadri2020}.
Motivated by the works of Mangasarian \cite{Mangasarian1977}, we present an equivalent form of ICP. 
\noindent  The article is presented as follows:   we present an equivalent form of ICP  and provide some  conditions for the solution of ICP in Section 2. The  conclusion is given in Section 3.

	%%%%%%%%%%%%%%%%%%%%%%%%%%%%%%%%%%%%%%%%%%%%%%%%%%%%%%%%%%%%%%%%%%%%
	\section{Main results}\label{sec2}
\noindent In the following, we provide an equivalent expression of the implicit complementarity problem.
\begin{thm}\label{thm1}
	Suppose  $\mathcal{A} \in \mathbb{R}^{n \times n}$   and $b \in \mathbb{R}^{n}$. Then $r^* \in \mathbb{R}^{n \times n }$  be the solution of   ICP$( \mathcal{A}, b, f)$ $\iff$ $\mathcal{R}(r^*)=0$, where  $\mathcal{R} : \mathbb{R} \rightarrow \mathbb{R}$ is defined as 
	\begin{equation}\label{E1}
		\mathcal{R}(r)=	\mathcal{H}(r)-(\mathcal{H}(r)-(\mathcal{A}r+b))_{+}.		
	\end{equation}
\end{thm}
\begin{proof} Suppose $\mathcal{R}(r^*)=0$, it follows that $(\mathcal{H}(r^*)-(\mathcal{H}(r^*)-(\mathcal{A}r^*+b))_{+})=0$. This implies that 
	\begin{equation}\label{eq1}
		\mathcal{H}(r^*)=(\mathcal{H}(r^*)-(\mathcal{A}r^*+b))_{+}.
	\end{equation}
	Component-wise, we  consider two cases:\\
	\textbf{Case 1.} when $\mathcal{H}_{i}(r^*) \geq (\mathcal{A}r^*+b)_{i}$, where $\mathcal{H}_{i}(r^*)$ denotes the $i^{th}$ component of the $\mathcal{H}(r^*)$. Then   Eq.  (\ref{eq1}) can be written as 
	\begin{equation*}
		\mathcal{H}_{i}(r^*)=\mathcal{H}_{i}(r^*)-(\mathcal{A}r^*+b)_{i}.
	\end{equation*}
	It follows that, $$\mathcal{F}_{i}(r^*)=(\mathcal{A}r^*+b)_{i}=0.$$\\
	\textbf{Case 2.}  when $\mathcal{H}_{i}(r^*) ~\textless ~(\mathcal{A}r^*+b)_{i}$ $\implies$  $((\mathcal{A}r^*+b)_{i}- \mathcal{H}_{i}(r^*))\ \textless ~0$. Then, we get  
	\begin{equation*}
		\mathcal{H}_{i}(r^*)=0.
	\end{equation*}
	From case (1) and case (2),  $\mathcal{H}_{i}(r^*)\mathcal{F}_{i}(r^*)=0 ~\forall ~ i $  $\implies$ $\mathcal{H}(r^*)^{T}\mathcal{F}(r^*)=0 $\\
	Conversely, let $r^*$ be the solution of  system (\ref{eq0}). By  complementary condition  of ICP,  either $\mathcal{H}_{i}(r^*)=0$ or $\mathcal{F}_{i}(r^*)=0 ~\forall ~i$.\\	
	Component-wise, we  consider two cases:\\
	\textbf{Case 1.} If $\mathcal{H}_{i}(r^*)=0$  and $\mathcal{F}_{i}(r^*) \textgreater 0$,  Eq. (\ref{E1}) becomes 
	$\mathcal{R}_{i}(r^*)=(-(-(\mathcal{A}r^*+b))_{+})$
	$\implies $ $\mathcal{R}_{i}(r^*)=0.$\\
	\textbf{Case 2.} If $\mathcal{H}_{i}(r^*) \textgreater 0$  and $\mathcal{F}_{i}(r^*) = 0$, then 
	$\mathcal{R}_{i}(r^*)=(\mathcal{H}_{i}(r^*)-(\mathcal{H}_{i}(r^*))_{+})$
	$\implies$ $\mathcal{R}_{i}(r^*)=0.$\\
	From case (1) and case (2), we get  $\mathcal{R}_{i}(r^*)=0 ~ \forall ~ i.$ Then  $\mathcal{R}(r^*)=0$
\end{proof}
\begin{rem} Let $\mathcal{S}(r)=\mathcal{H}(r)- \mathcal{R}(r)$, then 
	$r^*$ is the solution of  ICP$( \mathcal{A}, b, f )$ $\iff$  $\mathcal{S}(r^*)= \mathcal{H}(r^*)$.	
\end{rem}
\begin{lem}
	Suppose   $\mathcal{A} \in \mathbb{R}^{n \times n}$  and $b \in \mathbb{R}^{n}$. Let  $\Omega_{1}, \Omega_{2} \in \mathbb{R}^{n \times n}$  be  two positive diagonal matrices and  define the mapping $\bar{\mathcal{R}}(r)= (\Omega_{1}\mathcal{H}(r)- (\Omega_{1}\mathcal{H}(r)-\Omega_{2}(\mathcal{A}r+b))_{+})$. Then  $r^* $ is the solution of ICP$( \mathcal{A}, b, f )$ $\iff$   $\bar{\mathcal{R}}(r^*)=0$ 
\end{lem}
\begin{proof} Suppose $\bar{\mathcal{R}}(r^*)=0$, it follows that $$(\Omega_{1}\mathcal{H}(r^*)-(\Omega_{1}\mathcal{H}(r^*)-\Omega_{2}(\mathcal{A}r^*+b))_{+})=0.$$ Then  we write  
	\begin{equation}\label{eq2}
		\Omega_{1}\mathcal{H}(r^*)=(\Omega_{1}\mathcal{H}(r^*)-\Omega_{2}(\mathcal{A}r^*+b))_{+}.
	\end{equation}
	Component-wise, we  consider two cases:\\
	\textbf{Case 1.} when $(\Omega_{1}\mathcal{H}(r^*))_{i} \geq (\Omega_{2}(\mathcal{A}r^*+b))_{i}$, then   Eq.  (\ref{eq1}) becomes 
	\begin{equation*}
		(\Omega_{1}\mathcal{H}(r^*))_{i}=(\Omega_{1}\mathcal{H}(r^*))_{i}-(\Omega_{2}(\mathcal{A}r^*+b))_{i}.
	\end{equation*}
	Then,  
	$(\Omega_{2}(\mathcal{A}r^*+b))_{i}=0$ $\implies$
	$\mathcal{F}_{i}(r^*)=(\mathcal{A}r^*+b)_{i}=0$.\\
	\textbf{Case 2.}  when $(\Omega_{1}\mathcal{H}(r^*))_{i} ~\textless ~ (\Omega_{2}(\mathcal{A}r^*+b))_{i}$, this implies that $((\Omega_{2}(\mathcal{A}r^*+b))_{i}- (\Omega_{1}\mathcal{H}(r^*))_{i})\ \textless 0$. Then, we get  
	\begin{equation*}
		(\Omega_{1}\mathcal{H}(r^*))_{i}=0 \implies \mathcal{H}_{i}(r^*)=0,
	\end{equation*}
	from case (1) and case (2), we obtain  $\mathcal{H}_{i}(r^*)\mathcal{F}_{i}(r^*)=0 ~\forall ~i.  $ Hence,  $\mathcal{H}(r^*)^{T}\mathcal{F}(r^*)=0.  $\\
	Conversely, let $r^*$ be the solution of  Eq.  (\ref{eq0}). Then \\
	component-wise, we  consider two cases:\\
	\textbf{Case 1.} when $\mathcal{H}_{i}(r^*)=0$  and $\mathcal{F}_{i}(r^*) ~\textgreater~ 0$, then
	$\bar{\mathcal{R}}_{i}(r^*)=-(-(\Omega_{2}(\mathcal{A}r^*+b))_{i})_{+}$
	$\implies$ $\bar{\mathcal{R}}_{i}(r^*)=0.$\\
	\textbf{Case 2.} when $\mathcal{H}_{i}(r^*) ~\textgreater~ 0$  and $\mathcal{F}_{i}(r^*) = 0$, then 
	$\bar{\mathcal{R}}_{i}(r^*)=((\Omega_{1}\mathcal{H}(r^*))_{i}-((\Omega_{1}\mathcal{H}(r^*))_{i})_{+}).$\\
	Thus, $\bar{\mathcal{R}}_{i}(r^*)=0 ~\forall ~i.$
	From case (1) and case (2), we obtain  $\bar{\mathcal{R}}(r^*)=0.$
\end{proof}
\begin{prop}
	Suppose  $\mathcal{A} \in \mathbb{R}^{n \times n}$  and $b \in \mathbb{R}^{n}$. Let  $\delta : \mathbb{R} \rightarrow \mathbb{R}$ be any strictly increasing function  such that $\delta(0)=0$. Then   $r^* $ is the solution of  ICP$( \mathcal{A}, b, f )$ $\iff$ $\mathcal{G}(r^*)=0$,  $\mathcal{G}$ is the function from  $\mathbb{R}$ to  $\mathbb{R}$, given as 
	\begin{equation}\label{eq3}
		\mathcal{G}_{i}(r)=\delta(\lvert (\mathcal{A}r+ b )_{i}-\mathcal{H}_{i}(x)\rvert)-\delta((\mathcal{A}r+b)_{i})-\delta(\mathcal{H}_{i}(x)), ~~ i=1,2,\ldots,n.		
	\end{equation}
\end{prop}
\begin{proof} For some $i$, let $\mathcal{H}_{i}(r^*)~\textless~ 0$.  Then it follows that
	\begin{equation}\label{eq4}
		0 ~\textgreater ~\delta(\mathcal{H}_{i}(r^*))= \delta (\mathcal{F}_{i}(r^*)-\mathcal{H}_{i}(r^*))-\delta(\mathcal{F}_{i}(r^*)) \geq -\delta(\mathcal{F}_{i}(r^*)).
	\end{equation} 	
	Thus, $\mathcal{F}_{i}(r^*) \textgreater 0 $ and  $\mathcal{F}_{i}(r^*)-\mathcal{H}_{i}(r^*) ~\textgreater~ \mathcal{F}_{i}(r^*) ~ \textgreater~ 0$. This implies that  
	\begin{equation}\label{eq5}
		\delta (|\mathcal{F}_{i}(r^*)-\mathcal{H}_{i}(r)|)=  \delta (\mathcal{F}_{i}(r^*)-\mathcal{H}_{i}(r^*)) \textgreater \delta (\mathcal{F}_{i}(r)). \end {equation}
		From inequalities  (\ref{eq4}) and   (\ref{eq5}), we get
		$\mathcal{G}(r^*)~\textgreater~ 0$, this is the contradiction.\\
		If $\mathcal{H}_{i}(r^*)~\textgreater ~ 0$  and $\mathcal{F}_{i}(r^*) ~ \textgreater ~ 0$ for some $i$.
		Then, we consider two possibilities:\\
		1. when $\mathcal{H}_{i}(r^*) \textgreater \mathcal{F}_{i}(r^*)$, then 
		\begin{equation*}
			\delta (|\mathcal{F}_{i}(r^*)-\mathcal{H}_{i}(r^*)|)=  \delta (\mathcal{H}_{i}(r^*)-\mathcal{F}_{i}(r^*)) \textless \delta (\mathcal{H}_{i}(r^*)). \end {equation*}
			Then 
			$	\delta (|\mathcal{F}_{i}(r^*)-\mathcal{H}_{i}(r^*)|) - \delta (\mathcal{H}_{i}(r^*)) \textless  0$.
			This implies that 
			\begin{equation}\label{eq6}
				\begin{split}
					\mathcal{G}_{i}(r^*) \textless 0,		
				\end{split}
			\end{equation}
			2. when $\mathcal{F}_{i}(r^*) \textgreater \mathcal{H}_{i}(r^*)$, then 
			\begin{equation*}
				\delta (|\mathcal{F}_{i}(r^*)-\mathcal{H}_{i}(r^*)|)=  \delta (\mathcal{F}_{i}(r^*)-\mathcal{H}_{i}(r^*)) \textless \delta (\mathcal{F}_{i}(r^*)). \end {equation*}
				Then 
				$\delta (|\mathcal{F}_{i}(r^*)-\mathcal{H}_{i}(r^*)|) - \delta (\mathcal{F}_{i}(r^*)) \textless  0$.
				This implies that 
				\begin{equation}\label{eq7}
					\begin{split}
						\mathcal{G}_{i}(r^*) \textless 0.	
					\end{split}
				\end{equation}
				From inequalities (\ref{eq6}) and   (\ref{eq7}), we must have $\mathcal{G}_{i}(r^*) \textless 0$, again a contradiction.  Therefore, $r^*$ solve the   ICP$( \mathcal{A}, b, f)$.\\
				Conversely, let $r^* $ be the solution of   ICP$( \mathcal{A}, b, f)$, then  either $\mathcal{H}_{i}(r^*)=0$ or $\mathcal{F}_{i}(r^*)=0 ~\forall ~i$.\\
				Suppose $\mathcal{H}_{i}(r^*)=0$  and  $\mathcal{F}_{i}(r^*)\textgreater 0$, Then 
				\begin{equation*}
					\mathcal{G}_{i}(r^*)=\delta(\lvert \mathcal{F}_{i}(r^*) |)-\delta(\mathcal{F}_{i}(r^*)).
				\end{equation*} 
				This implies that $	\mathcal{G}_{i}(r^*)=0$.\\
				\noindent 	When $\mathcal{F}_{i}(r^*)=0$  and  $\mathcal{H}_{i}(r^*)\textgreater 0$, Then 
				\begin{equation*}
					\mathcal{G}_{i}(r^*)=\delta(\lvert  -\mathcal{H}_{i}(r^*)|)-\delta(\mathcal{H}_{i}(r^*)).
				\end{equation*} 
				This implies that $	\mathcal{G}_{i}(r^*)=0 ~ \forall~ i$. Therefore, $	\mathcal{G}(r^*)=0$
			\end{proof}	
\section{Conclusion} 
This article presents an equivalent formulation of the implicit complementarity problem. We introduce a equivalent  formulation of the implicit complementarity problem and demonstrate that solution of the equivalent  formulation is equal to the solution of the implicit complementarity problem. Moreover, By using strictly increasing function $\delta $, an  equivalent form of the implicit complementarity problem is provided. 
	
	\subsection*{Acknowledgments}
	The author, Bharat Kumar is thankful to the University Grants Commission (UGC), Government of India, under the SRF fellowship, Ref. No.: 1068/(CSIR-UGC NET DEC. 2017).
	%\section*{Authors' information}%% if any
	%Text for this section\ldots
	%%  After submission of the .TEX file,                     %%
	%%  you will be prompted to submit your .BBL file.         %%
	% if your bibliography is in bibtex format, use those commands:
	%\bibliographystyle{plain} % Style BST file (bmc-mathphys, vancouver, spbasic).
	%	\bibliography{bibfile}      % Bibliography file (usually '*.bib' )
	% for author-year bibliography (bmc-mathphys or spbasic)
	% a) write to bib file (bmc-mathphys only)
	% @settings{label, options="nameyear"}
	% or include bibliography directly:
	\bibliographystyle{plain}
	\bibliography{p11}

\end{document}